\documentclass[12pt]{amsart}
\usepackage{amscd}
\usepackage{amsmath}
\usepackage{verbatim}

\usepackage[all, cmtip,arrow]{xy}

\usepackage{pb-diagram,pb-xy}

\usepackage{amssymb}

\numberwithin{equation}{section}

\newtheorem{thm}[equation]{Theorem}
\newtheorem{prop}[equation]{Proposition}

\newtheorem{lemma}[equation]{Lemma}
\newtheorem{cor}[equation]{Corollary}

\theoremstyle{definition}
\newtheorem{rem}[equation]{Remark}
\newtheorem{example}[equation]{Example}
\newtheorem{dfn}[equation]{Definition}

\newcommand{\Steen}{S}
\newcommand{\sq}{\operatorname{sq}}
\newcommand{\st}{\operatorname{st}}

\newcommand{\onto}{\rightarrow\!\!\rightarrow}
\newcommand{\codim}{\operatorname{codim}}
\newcommand{\rdim}{\operatorname{rdim}}

\newcommand{\SB}{X}

\newcommand{\rad}{\mathop{\mathrm{rad}}}

\renewcommand{\Im}{\mathop{\mathrm{Im}}}

\newcommand{\ind}{\mathop{\mathrm{ind}}}

\newcommand{\rk}{\mathop{\mathrm{rk}}}
\newcommand{\CH}{\mathop{\mathrm{CH}}\nolimits}

\newcommand{\Ch}{\mathop{\mathrm{Ch}}\nolimits}

\newcommand{\res}{\mathop{\mathrm{res}}\nolimits}
\newcommand{\cores}{\mathop{\mathrm{cor}}\nolimits}
\newcommand{\tr}{\mathop{\mathrm{tr}}\nolimits}

\newcommand{\pr}{\operatorname{\mathit{pr}}}

\newcommand{\inc}{\operatorname{\mathit{in}}}

\newcommand{\Char}{\mathop{\mathrm{char}}\nolimits}

\newcommand{\Z}{\mathbb{Z}}
\newcommand{\F}{\mathbb{F}}

\newcommand{\HH}{\mathbb{H}}

\newcommand{\cF}{\mathcal F}
\newcommand{\cO}{\mathcal O}
\newcommand{\cT}{\mathcal T}
\newcommand{\cE}{\mathcal E}
\newcommand{\cA}{\mathcal A}
\newcommand{\cB}{\mathcal B}

\newcommand{\Spec}{\operatorname{Spec}}
\newcommand{\End}{\operatorname{End}}

\newcommand{\Aut}{\operatorname{Aut}}

\newcommand{\PS}{\mathbb{P}}

\newcommand{\Sum}{\operatornamewithlimits{\textstyle\sum}}

\newcommand{\compose}{\circ}

\newcommand{\Ker}{\operatorname{Ker}}

\newcommand{\CM}{\operatorname{CM}}

\renewcommand{\phi}{\varphi}

\newcommand{\mf}{\mathfrak}

\usepackage[hypertex]{hyperref}

\title
{Isotropy of unitary involutions}


\author
[N. Karpenko]
{Nikita Karpenko}

\author
[M. Zhykhovich]
{Maksim Zhykhovich}

\address
{Universit\'e Pierre et Marie Curie\\
Institut de Math\'ematiques de Jussieu \\
Paris\\
FRANCE}

\email
{karpenko {\it at} math.jussieu.fr,
{\it web page}: www.math.jussieu.fr/\~{ }karpenko}
\email
{zhykhovich {\it at} math.jussieu.fr,
{\it web page}: www.math.jussieu.fr/\~{ }zhykhovich}

\date
{5 May 2011. Revised: 2 April 2012. Final revision: 4 June 2012}

\begin{document}

\begin{abstract}
We prove the so-called {\em Unitary Isotropy Theorem}, a result on
isotropy of a unitary involution.
The analogous previously known results on isotropy of
orthogonal and symplectic involutions as well as on hyperbolicity of
orthogonal, symplectic, and unitary involutions are formal consequences of
this theorem.
A component of the proof is a detailed study of the
quasi-split unitary grassmannians.
\end{abstract}

\maketitle

\tableofcontents

\section
{Introduction}

Let
$K$ be an arbitrary field of characteristic different from $2$,
$A$ a central simple $K$-algebra,
$\tau$ an {\em involution} on $A$, i.e., a self-inverse ring
anti-automorphism of $A$.
Let $F\subset K$ be the subfield of $\tau$-invariant elements of $K$.
In this paper we prove a
general {\em Isotropy Theorem} for algebras with involution saying
that if $\tau$ becomes isotropic over any
field extension of $F$ splitting
$A$, then $\tau$ becomes isotropic over
some finite odd degree field extension of $F$.
More precisely, in the case of symplectic $\tau$, ``splitting'' has to be
replaced by ``almost splitting''.
In the general case,
note that by the example of \cite{MR1850658}, $\tau$ over
$F$ does not need to be isotropic even if it becomes isotropic over an odd degree field extension.

We refer to \cite{MR1632779} for generalities on central simple algebras
with involutions.
The involution $\tau$ is {\em isotropic}, if the algebra $A$ contains a
non-zero right ideal $I$ satisfying $\tau(I)\cdot I=0$.
The algebra $A$ is {\em split}, if
it is isomorphic to a full matrix algebra over $K$;
it is {\em almost split}, if it is split or isomorphic to a full matrix algebra over a quaternion
division $K$-algebra.

Roughly speaking, Isotropy Theorem provides a possibility
to split the algebra without harming too much the involution.
It is important because
it allows one to reduce questions on involutions on central simple
algebras to the case of the split algebra, where the notion of involution
is equivalent to a simpler notion of bilinear form.
An example of application of such reduction is given in \cite[Theorem
3.8]{canondim}.

Concerning the history,
we do not know who first raised the conjecture on Isotropy Theorem, but
the first named author learned it from A. Wadsworth during a conference in
the first half of 90s.
Numerous particular cases or relaxations of this conjecture has been studied and proved
since than.
One of them is {\em Hyperbolicity Theorem}:
if $\tau$ becomes {\em hyperbolic} over any
field extension of $F$ splitting $A$, then $\tau$ is hyperbolic over
$F$.
Hyperbolicity Theorem has been proved in \cite{hypernew-tignol} for the exponent $2$ case and
in \cite{unitary} for the unitary case.

For algebras $A$ of exponent $2$,
Isotropy Theorem has been proved in \cite{oddisotro-tignol}.
More precisely, it has been reduced to the case of orthogonal $\tau$ by
J.-P. Tignol and proved in the orthogonal case by the first named author.
In the remaining case, proved in Theorem \ref{main} of the present paper,
the involution $\tau$ is of unitary type so that
the field extension $K/F$ is of degree $2$.

Before starting discussion of the proof of Theorem \ref{main},
we would like to mention that the orthogonal and symplectic cases of Isotropy Theorem
are formal consequences of its unitary case -- Theorem \ref{main}.
This relationship is explained in \cite[\S5]{unitary} and
\cite{symple}.

The proof in the unitary case, made in Section \ref{Unitary isotropy theorem},
goes along the lines of the proof of the
orthogonal case, but there are at least two important differences.
First of all, the information on orthogonal grassmannians needed in the
orthogonal case was already available: partially from topology,
partially from more recent works of A. Vishik \cite{Vishik-u-invariant},
\cite{MR2148072} partially remaking in algebraic terms the available topological material.
In contrast with this, the needed information on unitary grassmannians was
not available.
Sections \ref{Chow ring of split unitary grassmannians}
and \ref{Steenrod operations for split unitary grassmannians} cover this
need.

To explain the second difference, we have to sketch the proof.
It is easily reduced to the case of $A$ of index $2^r$ with $r\geq1$.
Let $Y$ be the $F$-variety of isotropic right ideals in $A$ of reduced dimension
$2^r$.
Let $X$ be the $F$-variety of all right ideals in $D$ of reduced dimension
$2^{r-1}$, where $D$ is a central {\em division} $K$-algebra
Brauer-equivalent to $A$.
Considering Chow motives with coefficients in $\F_2:=\Z/2\Z$ of smooth projective
$F$-varieties,
we manage to show, that certain indecomposable direct summand of the motive of
$X$, namely, the so-called {\em upper} motive of $X$ introduced in \cite{upper}, is isomorphic to a direct
summand of the motive of $Y$.
The corresponding projector $\pi$ is a cycle class in the modulo $2$ Chow group $\Ch_{\dim Y}(Y\times
Y)$.
With some more effort, we come to the case where $\pi$ is {\em symmetric},
i.e., invariant under the factor exchange automorphism of the Chow group.
We finish by applying to $\pi$ a certain operation which transforms it
to a $0$-cycle class in $\Ch_0(Y\times Y)$ of degree $1$ modulo $2$ and therefore terminates the proof.

The shortest way to explain where the operation comes from is as follows.
By \cite{MR2377113}, the projector $\pi$ can be lifted to the algebraic
cobordism.
Then, applying an appropriate
symmetric operation of \cite{MR2332601} and
projecting back from cobordism to the Chow group, we get the required $0$-cycle class.

Fortunately, the symmetric operations and algebraic cobordism theory are
not really needed here and thus we are not restricted to the
characteristic $0$.
Actually, we succeed to compute the above symmetric operation {\em because}
on symmetric projectors it
can be described in terms of the Steenrod operations on the modulo $2$ Chow groups.
This is done in Section \ref{Operations sq and st}.
The need of Steenrod operations explains our characteristic assumption.

The needed operation is related with the difference of two other operations:
$\sq$, given by the
squaring, and $\st$, given by a Steenrod operation.
The proof succeeds if the value of one operation turns out to be trivial
and the value of the other one -- non-trivial, see Lemma \ref{symmop}.
The value $\sq(\pi)$ is computed due to its relation with the {\em rank} of
the motive, see Lemmas \ref{1.5} and \ref{new1.2};
the needed ranks are calculated in Section \ref{Some ranks of some
motives}.
To compute the value $\st(\pi)$ we use the information on the Steenrod
operations on quasi-split unitary grassmannians obtained in Section
\ref{Steenrod operations for split unitary grassmannians}.

The second difference between the orthogonal and the unitary cases is as follows:
$\st(\pi)$ is the trivial value in the orthogonal case while $\sq(\pi)$ is the trivial value in the unitary case.
In particular, we have to check the non-triviality of the more sophisticated $\st(\pi)$
here, which is certainly more difficult than to show its triviality.

\bigskip
\noindent
{\sc Acknowledgements.}
The authors thank Alexander Merkurjev for permission to include Lemma
\ref{AM} and Burt Totaro for information about the state of study of
unitary grassmannians in topology.

\section
{Operations $\sq$ and $\st$}
\label{Operations sq and st}

Let $F$ be a field of characteristic $\ne2$.
Let $X$ be a connected smooth projective variety over $F$.

We write $\CH$ for the integral Chow group and we write $\Ch$ for
the Chow group with coefficients in $\F_2$.

We will use the following observation due to A. Merkurjev:

\begin{lemma}
\label{AM}
Let $\delta\colon X\to X\times X$ be the diagonal morphism.
For any $\mf{a},\mf{b}\in\CH(X\times X)$ one has
$\deg(\mf{b}^t\cdot\mf{a})=\deg(\delta^*(\mf{b}\compose\mf{a}))$,
where $\cdot$ stands for the intersection product, $\compose$
stands for the correspondence product, and $^t$ stands for the transposition of correspondences.
\end{lemma}

\begin{proof}
By the following commutative diagram of pull-backs and push-forwards
$$
\begin{CD}
\CH(X\times X) @<e^*<< \CH(X\times X\times X) @<s^*<< \CH(X\times X\times X\times X)\\
@V\pr_{1*}VV @V\pr_{13*}VV \\
\CH(X) @<\delta^*<< \CH(X\times X).
\end{CD}
$$
where $\pr_1(x,y)=x$, $\pr_{13}(x,y,z)=(x,z)$,
$e(x,y)=(x,y,x)$, and $s(x,y,z)=(x,y,y,z)$,
we obtain
$$
\pr_{1*}(\mf{b}^t\cdot\mf{a})=\pr_{1*}e^*s^*(\mf{a}\times\mf{b})=
\delta^*\pr_{13*}s^*(\mf{a}\times\mf{b})=\delta^*(\mf{b}\compose\mf{a}),
$$
hence
$\deg(\mf{b}^t\cdot\mf{a})=\deg\pr_{1*}(\mf{b}^t\cdot\mf{a})=\deg\delta^*(\mf{b}\compose\mf{a})$.
\end{proof}

\begin{dfn}
Our first basic operation is
a map $\sq\colon \Ch(X\times X)\to\Z/4\Z$ defined as follows.
For any $\alpha\in\Ch(X\times X)$ we take its integral representative $\mf{a}\in\CH(X\times X)$
and set
$$
\sq(\alpha):=\deg(\mf{a}\cdot\mf{a})\mod{4}.
$$
Since any other integral representative of the same $\alpha$ is of the form
$\mf{a}+2\mf{b}$ with some $\mf{b}\in\CH(X\times X)$ and
$\deg((\mf{a}+2\mf{b})^2)\equiv\deg(\mf{a}^2)\pmod{4}$, the map $\sq$ is well-defined.

We also define an auxiliary
operation $\sq'\colon\Ch(X\times X)\to\Z/4\Z$ as follows.
For any $\alpha\in\Ch(X\times X)$ we take its integral representative $\mf{a}\in\CH(X\times X)$
and set $\sq'(\alpha):=\deg(\mf{a}^t\cdot\mf{a})\mod{4}$.
Since any other integral representative of the same $\alpha$ is of the form
$\mf{a}+2\mf{b}$ with some $\mf{b}\in\CH(X\times X)$ and
$\deg((\mf{a}+2\mf{b})^t\cdot(\mf{a}+2\mf{b}))\equiv\deg(\mf{a}^2)\pmod{4}$,
because
$$
\deg(\mf{b}^t\cdot\mf{a})=\deg\big((\mf{b}^t\cdot\mf{a})^t\big)=\deg(\mf{a}^t\cdot\mf{b}),
$$
the map $\sq'$ is well-defined as well.
\end{dfn}

\begin{lemma}
\label{new1.2}
Let $\sq$ and $\sq'$ be the introduced operations.
Then
\begin{enumerate}
\item
$\sq'(\alpha)=\sq(\alpha)$ for any symmetric projector $\alpha$;
\item
$\sq'(\alpha+\beta)=\sq'(\alpha)+\sq'(\beta)$
for any orthogonal correspondences $\alpha$ and $\beta$;
\item
$\sq'(\alpha)_E=\sq'(\alpha_E)$ and $\sq(\alpha)_E=\sq(\alpha_E)$
for any field extension $E/F$ and any $\alpha$.
\end{enumerate}
\end{lemma}

\begin{proof}
(1)
Indeed, such $\alpha$ has a symmetric integral representative:
if $\mf{a}$ is any integral representative, then $\mf{a}^t\compose\mf{a}$
is a symmetric integral representative of $\alpha$.
Computing $\sq(\alpha)$ and $\sq'(\alpha)$ with the help of a symmetric integral representative
of $\alpha$, we get the same.

(2)
Let $\mf{a},\mf{b}\in\CH(X\times X)$ be integral; representatives of $\alpha,\beta$.
It suffices to show that $\deg(\mf{b}^t\cdot\mf{a})\equiv 0\pmod{2}$.
By Lemma \ref{AM}, $\deg(\mf{b}^t\cdot\mf{a})=\deg(\delta^*(\mf{b}\compose\mf{a}))$.
Since the correspondences $\beta$ and $\alpha$ are orthogonal,
$\mf{b}\compose\mf{a}\in2\CH(X\times X)$.

(3)
Trivial.
\end{proof}

We are working with the Chow motives over $F$ with coefficients in $\F_2$, \cite[Chapter XII]{EKM}.
A motive is {\em split}, if it is isomorphic to a (finite) direct sum of Tate
motives.
A motive is {\em geometrically split}, if it splits over a field extension
of $F$.
The {\em rank} $\rk M$ of a geometrically split motive $M$ is the number
of Tate summands in the decomposition of $M_E$ for a field extension $E/F$
such that $M_E$ is split (this number does not depend on the choice of
$E$).
If $\alpha$ is a projector on a smooth projective variety $X$ such that
the motive $(X,\alpha)$ is geometrically split, we set
$\rk\alpha:=\rk(X,\alpha)$.

\begin{lemma}
\label{1.5}
Let $\alpha\in\Ch(X\times X)$ be a projector
and assume that the motive
$(X,\alpha)$ is geometrically split (so that the rank $\rk(\alpha)\in\Z$ of $\alpha$ is defined).
Then $\sq'(\alpha)=\rk(\alpha)\mod 4$.
\end{lemma}

\begin{proof}
By the naturality of $\sq'$ (Lemma \ref{new1.2}),
we may assume that the motive $(X,\alpha)$ is split, that is, that
$(X,\alpha)$ is isomorphic to a finite sum of Tate motives.
The number of the summands is the rank.
By additivity of $\sq'$ (Lemma \ref{new1.2}), we may assume that the rank is $1$.
In this case $\alpha$ has an integral representative of the form $a\times b$
with some homogeneous
$a,b\in\CH(X)$ having odd $\deg (a\cdot b)$.
It follows that $\sq'(\alpha)=1\mod{4}$.
\end{proof}

Now we are going to define our second basic operation.
Consider
the total cohomological Steenrod operation $\Steen^\bullet$, \cite[Chapter XI]{EKM}.
This is a certain endomorphism of the cofunctor $\Ch$ of the category of smooth
$F$-varieties to the category of rings.


\begin{lemma}
\label{NK}
For any $\alpha,\beta\in\Ch(X\times X)$ one has
$$
\deg\Steen^\bullet(\beta^t\compose\alpha)=
\deg\big(\pr_{2*}(\Steen^\bullet(\alpha))\cdot\pr_{2*}(\Steen^\bullet(\beta))\cdot c_\bullet(-T_X)\big),
$$
where $T_X$ is the tangent bundle of $X$, $c_\bullet$ is the total
(modulo $2$) Chern class, $\pr_2\colon X\times X\to X$ is the projection
onto the second factor, and $\deg\colon\Ch\to\F_2$ is the degree
homomorphism modulo $2$.
\end{lemma}

\begin{proof}
Let
$\pr_{13}, \pr_{23}\colon X \times X \times X \rightarrow X \times X$,
and $\pr_{1}\colon X \times X \rightarrow X$
be the projections
$$
(x,y,z)\mapsto(x,z),\;(x,y,z)\mapsto(y,z),\;\text{ and }\; (x,y)\mapsto x.
$$
Since $\beta^t\compose\alpha =
\pr_{13*}\big((\alpha \times [X]) \cdot ([X] \times \beta^t)\big)$, we have
$$\Steen^\bullet(\beta^t\compose\alpha) =
\pr_{13*}\big((\Steen^\bullet(\alpha) \times [X])
\cdot ([X] \times \Steen^\bullet(\beta^t)) \cdot ([X] \times c_\bullet(-T_X) \times [X]) \big)\,.
$$
Note that $\deg \compose \pr_{13*} = \deg \compose \pr_{1*} \compose \pr_{23*}$.
By projection formula, we have
\begin{multline*}
 \pr_{23*}\big((\Steen^\bullet(\alpha) \times [X]) \cdot ([X]
 \times \Steen^\bullet(\beta^t)) \cdot ([X] \times c_\bullet(-T_X) \times [X]) \big) = \\
\Steen^\bullet(\beta^t) \cdot \big(\pr_{2*} (\Steen^\bullet(\alpha))
\times [X]\big) \cdot  \big(c_\bullet(-T_X) \times [X]\big)
\end{multline*}
and
\begin{multline*}
$$\pr_{1*} \Big(\Steen^\bullet(\beta^t) \cdot \big(\pr_{2*} (\Steen^\bullet(\alpha))\times [X]\big)
\cdot  \big(c_\bullet(-T_X) \times [X]\big)\Big) = \\
\pr_{2*}(\Steen^\bullet(\alpha))\cdot\pr_{2*}(\Steen^\bullet(\beta))\cdot c_\bullet(-T_X)\, .
\end{multline*}
Therefore
\begin{equation*} \deg\Steen^\bullet(\beta^t\compose\alpha)=
\deg\big(\pr_{2*}(\Steen^\bullet(\alpha))\cdot\pr_{2*}(\Steen^\bullet(\beta))\cdot c_\bullet(-T_X)\big) \, .
\qedhere
\end{equation*}
\end{proof}

\begin{dfn}
We define our second basic operation
$\st\colon\Ch(X\times X)\to\Z/4\Z$ as follows.
For any $\alpha\in\Ch(X\times X)$ we choose an integral representative
$\mf{a}\in\CH(X\times X)$
of
$\Steen^\bullet(\alpha)\in\Ch(X\times X)$ and
set
$$
\st(\alpha):=\deg\big(\pr_{2*}(\mf{a})^2\cdot
c_\bullet(-T_X)\big)\pmod{4},
$$
where $c_\bullet$ refers now to the {\em integral} total Chern class.
Clearly, the map $\st$ is well-defined because the choice of $\mf{a}$ does not
affect the resulting value.
\end{dfn}


\begin{lemma}
\label{new st}
Let $\st$ be the introduced operation.
Then
\begin{enumerate}
\item
$
\st(\alpha)\mod{2}=\deg\Steen^\bullet(\alpha^t\compose\alpha)
$ for any $\alpha$;
in particular,
$
\st(\alpha)\mod{2}=\deg\Steen^\bullet(\alpha)
$
if the correspondence $\alpha$ is a symmetric projector;
\item
$
\st(\alpha+\beta)=\st(\alpha)+\st(\beta)
$
for any correspondences $\alpha$ and $\beta$ such that
$\deg\Steen^\bullet(\beta^t\compose\alpha)=0$;
in particular, the additivity formula holds for orthogonal symmetric
correspondences $\alpha,\beta$;
\item
$\st(\alpha)_E=\st(\alpha_E)$
for any field extension $E/F$ and any $\alpha$.
\end{enumerate}
\end{lemma}

\begin{proof}
(1)
This is the particular case $\beta=\alpha$ of Lemma \ref{NK}.

(2)
Let $\mf{a},\mf{b}\in\CH(X\times X)$ be integral representatives of
$\Steen^\bullet(\alpha),\Steen^\bullet(\beta)$.
Then $\mf{a}+\mf{b}$ is an integral representative of $\Steen^\bullet(\alpha+\beta)$
and it suffices to show that
$$
\deg\big(\pr_{2*}(\mf{a})\cdot\pr_{2*}(\mf{b})\cdot c_\bullet(-T_X)\big)\equiv 0\pmod{2}.
$$
This is indeed so by Lemma \ref{NK} and the condition on $\alpha,\beta$.

(3)
Trivial.
\end{proof}

The two operations $\sq$ and $\st$ are related as follows:

\begin{lemma}
\label{symmop}
Let $d=\dim X$.
For any symmetric projector $\alpha\in\Ch^d(X\times X)$ one has
$\sq(\alpha)\equiv\st(\alpha)\pmod{2}$.
If moreover $X$ has no closed points of odd degree, then $\sq(\alpha)=\st(\alpha)$.
\end{lemma}

\begin{proof}
The value $\sq(\alpha)$ is given by the degree of certain integral representative $\mf{a}$ of
$\alpha^2$.
By Lemma \ref{new st},
the value $\st(\alpha)$ is given by the degree of certain integral representative $\mf{b}$ of
$\Steen^d(\alpha)$.
Since $\Steen^d(\alpha)=\alpha^2$, it follows that $\mf{a}-\mf{b}\in2\CH(X\times
X)$.
Since $\deg\CH(X\times X)=\deg\CH(X)$, we get that
$\deg\mf{a}-\deg\mf{b}\in2\deg\CH(X)$.
In particular, $\deg\mf{a}-\deg\mf{b}\in4\Z+2\deg\CH(X)$ as claimed.
\end{proof}

\begin{rem}
Lemma \ref{symmop} in particular says that the difference $\st-\sq$ restricted to the set
of symmetric projectors $\mathrm{SP}\subset\Ch^d(X\times X)$, is divisible by $2$.
The resulting map
$$
(\st-\sq)/2\colon\mathrm{SP}\to\Z/2\Z
$$
can be viewed as a replacement for a certain symmetric operation.
One advantage of this replacement is that it works over an arbitrary field of characteristic
$\ne2$.
Note that symmetric operations are defined only over fields of characteristic $0$.
\end{rem}

\section
{Chow ring of quasi-split unitary grassmannians}
\label{Chow ring of split unitary grassmannians}

Let $K$ be a field of arbitrary characteristic.
Let $V$ be a finite-dimensional vector space over $K$.
We set $n:=\dim V$.
Although in relation with our main purpose we are only interested in the case of even $n$,
we treat the case of odd $n$ because it differs from the even one only in a few places.
For any subset $I\subset\{1,2\dots,[n/2]\}$,
where $[n/2]=n/2$ for even $n$ and $[n/2]=(n-1)/2$ for odd $n$,
we write $G_I(V)$ for the
variety of flags of subspaces in $V$ of dimensions given by $I$.
In particular, for any integer $k\in\{1,\dots,[n/2]\}$, the variety
$G_{\{k\}}(V)$,
which we simply denote as $G_k(V)$, is the grassmannian of $k$-planes.

Let us consider the closed subvariety $H_I=H_I(V)$ of the product $G_I(V)\times
G_I(V^*)$, where $V^*$ is the dual vector space of $V$, defined by the orthogonality condition:
$H_I$ is the variety of pairs of flags such that each space of the first
flag is orthogonal to the corresponding space of the second
flag (or, equivalently, the biggest space of the first flag is orthogonal to the biggest
space of the second flag).

\begin{example}
The variety $H_k$ is the variety of pairs of $k$-planes
$U\subset V$, $U'\subset V^*$ such that $U\cdot U'=0$.
It is canonically isomorphic to the variety of flags in $V$ consisting of a $k$-plane
contained in a ($n-k$)-plane.
\end{example}

We fix now a non-degenerate symmetric bilinear form $\mf{b}$ on $V$ which gives
a self-dual isomorphism $V \simeq V^*$.
This endows the variety $G_I(V)\times G_I(V^*)$ with a switch involution
and the subvariety $H_I$ is stable under it.
The induced involution on
$\CH(H_I)$ will be denoted by $\sigma$.
We are going to show that $\sigma$ does not depend on the choice of $\mf{b}$.

\begin{lemma}
The involution on $\CH(G_I(V)\times G_I(V^*))$ induced by the switch
involution on $G_I(V)\times G_I(V^*)$ given by $\mf{b}$ does not depend on $\mf{b}$.
\end{lemma}

\begin{proof}
The group $\Aut(V)\times\Aut(V^*)$ acts trivially on $\CH(G_I(V)\times
G_I(V^*))$, \cite[Corollary 4.2]{gug}.
\end{proof}

\begin{lemma}
The ring $\CH(H_I)$ is generated by the Chern classes of pull-backs
of the tautological bundles on $G_I(V)\times G_I(V^*)$.
\end{lemma}

\begin{proof}
Using projection on components, decompose the structure morphism $H_I\to\Spec
K$ into a chain of grassmannian bundles and apply \cite[Proposition 14.6.5]{Fulton}.
\end{proof}

\begin{cor}
The involution $\sigma$ on $\CH(H_I)$ does not depend on $\mf{b}$.
It is the unique involution for which the diagram
$$
\begin{CD}
\CH(G_I(V)\times G_I(V^*)) @>>> \CH(G_I(V)\times G_I(V^*))\\
@VVV @VVV\\
\CH(H_I) @>>> \CH(H_I)
\end{CD}
$$
commutes.
\qed
\end{cor}

We are going to study the subring $\CH(H_I)^\sigma\subset\CH(H_I)$ of the
$\sigma$-invariant elements.
Why we are interested in this subring is explained in Example \ref{link}.
More precisely, we will study the quotient of this subring by its ``elementary part'' --
the norm ideal $(1+\sigma)\CH(H_I)$.
We are basically interested in the case of $\#I=1$.

\subsection{$I=\{1\}$}

We start with the case $I=\{1\}$.
Note that the varieties $G_1(V)$ and $G_1(V^*)$ are projective spaces of dimension $n-1$.
If we choose a basis of $V$ and use the dual basis of $V^*$, then we identify $G_1(V)$ and $G_1(V^*)$ with
$\PS^{n-1}$, and $H_1$ becomes the
hypersurface in $\PS^{n-1}\times\PS^{n-1}$ given by the equation
$\sum_{i=1}^n x_iy_i = 0$, where
$x_i$ and $y_i$ are the respective homogeneous coordinates.
Such a hypersurface is known under the name {\em Milnor hypersurface},
\cite[\S2.5.3]{MR2286826}.

Let $h$ be the hyperplane class in $\CH^1(G_1(V))$ or in $\CH^1(G_1(V^*))$
and let us define the elements $a,b\in\CH^1(H_1)$ as the
pull-backs of $h\times 1$ and $1\times h$.
For any $i\geq0$, one has $a^i=c_i(-\mathcal{A})$ and $b^i=c_i(-\mathcal{B})$, where
$\mathcal{A}$ and $\mathcal{B}$ are the corresponding tautological vector bundles on $H_1$.

The ring $\CH(H_1)$ is generated by the two elements $a,b$ subject to the
relations $a^n=0$ and $a^{n-1}+a^{n-2}(-b)+\dots +(-b)^{n-1}=0$
(implying $b^n=0$).
The involution $\sigma$ exchanges the generators $a$ and $b$.

Let $\cT$ be the vector bundle on $H_1$ whose fiber over a point $(U,U')$ is
$U\oplus U'$ (i.e., $\cT=\mathcal{A}\oplus \mathcal{B}$).
We write $c_i$ for $c_i(-\cT)$.
We have $c_i=a^i+a^{i-1}b+\dots+b^i\in\CH(H_1)^\sigma$.

The elements $c_{n-1},c_n,\dots$ are divisible by $2$.
Indeed,
$$
c_{n-1}/2=a^{n-1}+a^{n-3}b^2+\dots+ab^{n-2}=a^{n-2}b+\dots+b^{n-1}
$$
and $c_{n-1+i}/2=(c_{n-1}/2)\cdot a^i=(c_{n-1}/2)\cdot b^i$ for any $i\geq0$.

\begin{lemma}
\label{1}
The ring $\CH(H_1)^\sigma/(1+\sigma)\CH(H_1)$
is {\em additively} generated by the classes of the following elements:
$$
c_0,c_1,\dots,c_{n-2}\;\text{ and }\;c_{n-1}/2,c_n/2,\dots.
$$
Moreover, for any odd $i\leq n-2$ the class of $c_i$ is $0$, for any
even $i\geq n-1$ the class of $c_i/2$ is $0$, and for any $i>2n-3$ the class of $c_i/2$ is $0$.
\end{lemma}

\begin{proof}
The group $\CH^{<n-1}(H_1)$ is freely generated by $a^ib^j$ with
$i+j<n-1$.
Therefore the quotient $\CH(H_1)^\sigma/(1+\sigma)\CH(H_1)$ in
codimensions $<n-1$ is (additively) generated by the classes of
$a^ib^i$ ($2i<n-1$) which are also represented by $c_{2i}$.

For any $i=n-1,n,\dots,2n-3$, the group $\CH^i(H_1)$ is generated by
the elements
$$
a^{n-1}b^{i-(n-1)},a^{n-2}b^{i-(n-2)},\dots,a^{i-(n-1)}b^{n-1}
$$
whose alternating sum is $0$,
and this is the
only relation on the generators.
The quotient of the subgroup of $\sigma$-invariant elements by the norms
is therefore trivial for even $i$ and generated by the class of
$c_i/2$ for odd $i$.

Finally, for $i>2n-3=\dim H_1$, the group $\CH^i(H_1)$ is trivial.
\end{proof}

\begin{rem}
\label{complete}
Here is a complete analysis of the graded
ring
$$
R:=\CH(H_1)^\sigma/(1+\sigma)\CH(H_1),
$$
which is now easily done.
Similarities as well as differences with the Chow ring of a split
projective quadric are striking.

In the case of even $n$, the ring $R$
is generated (as a ring) by two elements: (the classes of) $ab\in R^2$ and $c:=c_{n-1}/2\in
R^{n-1}$ ($R^2$ and $R^{n-1}$ are the graded components of $R$).
The relations are: $(ab)^{n/2}=0$ and $c^2=0$.
The non-zero homogeneous elements of $R$ are as follows:
$$
(ab)^i=c_{2i},\;\;
c(ab)^i=c_{n-1+2i}/2,\;\;
\text{with }i=0,1,\dots,(n-2)/2.
$$

If $n$ is odd, the ring $R$
is generated by two elements: (the classes of) $ab\in R^2$ and $c:=c_n/2\in
R^n$.
The relations are: $(ab)^{(n-1)/2}=0$ and $c^2=0$.
The non-zero homogeneous elements of $R$ are as follows:
$$
(ab)^i=c_{2i},\;\;
c(ab)^i=c_{n+2i}/2,\;\;
\text{with }i=0,1,\dots,(n-3)/2.
$$

The geometric description of the generators (for arbitrary parity of $n$) is as follows.
The element $(ab)^i$ is the pullback of $h^i\times h^i\in\CH^{2i}\big(G_1(V)\times
G_1(V^*)\big)$.
To describe $c(ab)^i$, we take some
orthogonal subspaces $U\subset V$, $U'\subset V^*$ of dimension
$[n/2]-i$. Then $c(ab)^i$ is the class of
the (closed) subvariety $L_i\subset H_1$ of pairs of lines: one
line in $U$, the other in $U'$.
\end{rem}

\subsection{$I=\{k\}$}

Now we start to study the case of $I=\{k\}$ where $k$ satisfies $1\leq k\leq [n/2]$.
Although the ring $\CH(H_k)$ can be easily described by generators and
relations and $\sigma$ can be easily described in terms of the generators,
we do not know an easy way to understand $\CH(H^k)^\sigma$.

We write $\cT_k$ for the vector bundle on $H_k$ whose fiber over a point $(U,U')$ is
$U\oplus U'$ (in particular, $\cT_1=\cT$).

We consider the natural projections $\pi_1\colon H_{\{1,k\}}\to H_1$ and
$\pi_k\colon H_{\{1,k\}}\to H_k$.

\begin{lemma}[{cf. \cite[Proposition 2.1]{Vishik-u-invariant}}]
\label{1.6}
For any integer $i$ one has
$$
c_i(-\cT_k)=(\pi_k)_*\pi_1^*c_{i+2(k-1)}(-\cT_1).
$$
\end{lemma}

\begin{proof}
For any smooth scheme $X$ with a rank $k$ vector bundle $\cE$ one has
$$
c_i(-\cE)=\pi_*c_{i+k-1}(-\cO(-1)),
$$
where $\pi$ is the morphism
$\PS(\cE)\to X$ and $\cO(-1)$ is the tautological (line) bundle on $\PS(\cE)$.
If now $\cE_1,\cE_2$ are two rank $k$ vector bundles on $X$ and
$\pi$ is the morphism $\PS(\cE_1)\times_X\PS(\cE_2)\to X$, we get that
$$
c_i(-(\cE_1\oplus\cE_2))=\pi_*c_{i+2(k-1)}\big(-(\cO_1(-1)\oplus\cO_2(-1))\big).
$$
In particular, taking $\pi=\pi_k$, we see that
$$
c_i(-\cT_k)=(\pi_k)_*c_{i+2(k-1)}\big(-(\cO_1(-1)\oplus\cO_2(-1))\big).
$$
Since $\cO_1(-1)\simeq\pi_1^*\cA$ and $\cO_2(-1)\simeq\pi_1^*\cB$, we are done.
\end{proof}

\begin{cor}
The $\sigma$-invariant elements
$$
c_{n-2k+1}(-\cT_k),c_{n-2k+2}(-\cT_k),\dots,c_{2n-2k-1}(-\cT_k) \in\CH(H_k)^\sigma
$$
are divisible by $2$.
\qed
\end{cor}

We consider the projections
$\pi_k\colon H_{\{k,k+1\}}\to H_k$ and $\pi_{k+1}\colon H_{\{k,k+1\}}\to H_{k+1}$.
The vector bundle $\pi_k^*\cT_k$ is a subbundle of the vector bundle $\pi_{k+1}^*\cT_{k+1}$
(the quotient is a direct sum of two line bundles), and
we write $\alpha\in\CH^2(H_{\{k,k+1\}})$ for
$c_2(\pi_{k+1}^*\cT_{k+1}/\pi_k^*\cT_k)$.

\begin{lemma}[{cf. \cite[Lemma 2.6]{Vishik-u-invariant}}]
\label{1.8}
For $i\in\{0,1,\dots,n-2k\}$ one has
$$
\pi_k^*c_i(-\cT_k)\equiv\pi_{k+1}^*c_i(-\cT_{k+1})+\alpha\cdot c_{i-2}(-\cT_{k+1})\pmod{1+\sigma}.
$$
For
$i\geq n-2k+1$
one has
$$
\pi_k^*c_i(-\cT_k)/2\equiv\pi_{k+1}^*c_i(-\cT_{k+1})/2+\alpha\cdot c_{i-2}(-\cT_{k+1})/2\pmod{1+\sigma}.
$$
\end{lemma}

\begin{proof}
We play with the following commutative diagram
\begin{equation*}
\divide\dgARROWLENGTH by2
\begin{diagram}
\node{}\node{H_{\{1,k\}}}\arrow{sw,t}{1}
\arrow{se,t}{3}\node{}\\
\node{H_1}\node{H_{\{1,k,k+1\}}}\arrow{n,l}{2}\arrow{sw,l}{5}
\arrow{s,l,J}{6}\arrow{se,l}{7}
\node{H_k}\\
\node{H_{\{1,k+1\}}}\arrow{se,r}{11}\arrow{n,l}{4}
\node{H}\arrow{w,t}{9}\arrow{e,t}{10}
\node{H_{\{k,k+1\}}}\arrow{sw,r}{12}\arrow{n,r}{8}\\
\node{}\node{H_{k+1}}
\node{}
\end{diagram}
\end{equation*}
where $H$ is defined as the fiber product of $H_{\{1,k+1\}}$ and
$H_{\{k,k+1\}}$ over $H_{k+1}$.
The variety $H_{\{1,k,k+1\}}$ is naturally a closed subvariety (of
codimension $2$) in $H$ and 6 is the closed imbedding.
Note that $\pi_k=8$ and $\pi_{k+1}=12$.
By Lemma \ref{1.6}, the elements $c_i(-\cT_k)$ and $c_i(-\cT_k)/2$ are
$3_*1^*(x)$ for certain $x\in\CH(H_1)^\sigma$.
Let us compute $y:=8^*3_*1^*(x)$ for an arbitrary $x\in\CH(H_1)^\sigma$.

The square 3-8-7-2 is transversal cartesian.
Therefore $8^*3_*=7_*2^*$.
By commutativity of the square 1-2-5-4, $2^*1^*=5^*4^*$ so that
$y=7_*5^*4^*(x)$.
By commutativity of the triangles 5-6-9 and 6-7-10,
$y=10_*6_*6^*9^*4^*(x)=10_*[H_{\{1,k,k+1\}}]\cdot9^*4^*(x)$.
The class $[H_{\{1,k,k+1\}}]\in\CH^2(H)$ is computed modulo $1+\sigma$
as $9^*4^*(ab)+10^*(\alpha)$.
It follows that $y\equiv10_*9^*4^*(abx)+\alpha\cdot10_*9^*4^*(x)$.
Since the square 9-10-12-11 is transversal cartesian, $10_*9^*=12^*11_*$,
so that we finally get
$y\equiv12^*11_*4^*(abx)+\alpha\cdot12^*11_*4^*(x)\pmod{1+\sigma}$.

We get the first (resp. second) desired congruence taking $x=c_{i+2(k-1)}(-\cT_1)$
(resp. $x=c_{i+2(k-1)}(-\cT_1)/2$) by Lemma
\ref{1.6}, because $abc_{i+2(k-1)}(-\cT_1)\equiv c_{i+2(k+1)}(-\cT_1)\pmod{1+\sigma}$
(resp. $abc_{i+2(k-1)}(-\cT_1)/2\equiv c_{i+2(k+1)}(-\cT_1)/2\pmod{1+\sigma}$)
for the corresponding values of $i$ (cf. Remark \ref{complete}).
\end{proof}

\begin{prop}[{cf. \cite[Proposition 2.11]{Vishik-u-invariant}}]
\label{1.9}
The ring $\CH(H_k)^\sigma$  is generated (as a ring) modulo
the ideal $(1+\sigma)\CH(H_k)$ by
the elements
$c_i(-\cT_k)$ with even $i$ satisfying $0\leq i\leq n-2k$ and the elements
$c_i(-\cT_k)/2$ with odd $i$ satisfying $n-2k+1\leq i\leq 2n-2k-1$.
\end{prop}

\begin{proof}
For each integer $l$ with $1\leq l\leq k$, we consider the projection
$\pi_l\colon H_{\{1,\dots,k\}}\to H_l$ and the elements
\begin{multline}
\tag{$*$}
\pi_l^*c_i(-\cT_l) \text{ with even $i$ satisfying $0\leq i\leq n-2l$ and }\\
\pi_l^*c_i(-\cT_l)/2 \text{ with odd $i$ satisfying $n-2l+1\leq i\leq 2n-2l-1$.}
\end{multline}

\begin{lemma}[{cf. \cite[Lemma 2.12]{Vishik-u-invariant}}]
\label{1.10}
The ring $\CH(H_{\{1,\dots,k\}})^\sigma$ is generated modulo the ideal
$
(1+\sigma)\CH(H_{\{1,\dots,k\}})
$
by the elements ($*$) (with $l$ running over
$1,\dots,k$).
\end{lemma}

Before we prove Lemma \ref{1.10}, we have to explain the link to hermitian forms:
\begin{example}
\label{link}
Assume that the field $K$ is separable quadratic over some subfield $F\subset K$ and let $h$ be a
$K/F$-hermitian form on $V$.
Let $Y_I$ be the flag variety of totally isotropic subspaces in $V$.
The $K$-variety $(Y_I)_K$ is canonically isomorphic to $H_I$.
For any $k=1,\dots,[n/2]$, the identification $H_k=(Y_k)_K$
transforms $\cT_k$ to the tautological vector bundle on $(Y_k)_K$,
defined over $F$.
The non-trivial automorphism of $K/F$ induces an automorphism of
$\CH(Y_I)_K$ identified with $\sigma$.
The image of the change of field homomorphism $\CH(Y_I)\to\CH(H_I)$ is
contained in the subring $\CH(H_I)^\sigma\subset\CH(H_I)$ of the
$\sigma$-invariant elements.
Moreover, if $h$ is hyperbolic, the change of field homomorphism $\CH(Y_I)\to\CH(H_I)$ is
injective and its image coincides with $\CH(H_I)^\sigma$ so that we have a
canonical identification $\CH(Y_I)=\CH(H_I)^\sigma$;
the ideal $(1+\sigma)\CH(H_I)\subset\CH(H_I)^\sigma$ coincides with the
image of the norm homomorphism $\CH(Y_I)_K\to\CH(Y_I)$.
The statement on hyperbolic $h$ is a consequence of the motivic
decomposition of the motive of a projective homogeneous variety under a
{\em quasi-split} semisimple affine algebraic group obtained in
\cite{MR2110630}.
\end{example}

\begin{proof}[Proof of Lemma \ref{1.10}]
Since the statement does not depend on the base field $K$, we may assume
that $K$ is quadratic separable over some subfield $F$.
Then we fix a hyperbolic $K/F$-hermitian form on $V$ and replace $\CH(H_{\{1,\dots,k\}})^\sigma$ by
$\CH(Y_{\{1,\dots,k\}})$ (see Example \ref{link}).

We do induction on $k$.
The case $k=1$ is Lemma \ref{1}.
To pass from $k-1$ to $k$,
we apply
\cite[Lemma 5.6]{gug},
a variant of \cite[Statement 2.13]{Vishik-u-invariant}.
Let $Y\to X$ be the projection
$Y_{\{1,2,\dots,k\}}\to Y_{\{1,2,\dots,k-1\}}$ and let $B$ be the subgroup of $\CH(Y)$
generated by the norms
and the elements ($*$) with $l=k$.
We have to show that $B=\CH(Y)$ and \cite[Lemma 5.6]{gug}
tells us that it suffices to verify two following conditions:
\begin{itemize}
\item[(a)]
for the generic point $\theta\in X$, the composition $B\hookrightarrow\CH(Y)\onto\CH(Y_\theta)$
is surjective;
\item[(b)]
for any point $x\in X$, at least one of two holds:
\begin{itemize}
\item[(b1)]
the specialization $\CH(Y_\theta)\to\CH(Y_x)$ is
surjective;
\item[(b2)]
for the filtration on $\CH(Y)$ whose $i$th term $\cF^i\CH(Y)$ is the
subgroup generated by the classes of cycles on $Y$ with image in $X$ of
codimension $\geq i$, and for $r:=\codim x$,
the image of
$\CH(Y_x)\to \cF^r\CH(Y)/\cF^{r+1}$ is
in the subgroup of classes of elements of $B\cap\cF^r\CH(Y)$.
\end{itemize}
\end{itemize}

Each fiber of our morphism $Y\to X$ is a hermitian quadric given by a hyperbolic hermitian space
of dimension $n-2(k-1)$.
Let us check that Condition (a) holds.
The restriction of $\pi_k^*\cT_k$ to the generic fiber of the projection is
isomorphic to the direct sum of $\cT_1$ and a trivial vector bundle of rank $2(k-1)$.
Therefore the pull-backs of the elements ($*$) to the generic fiber give
the elements
\begin{multline*}
c_i(-\cT_k) \text{ with even $i$ satisfying $0\leq i\leq n-2k$ and }\\
c_i(-\cT_k)/2 \text{ with odd $i$ satisfying $n-2k+1\leq i\leq 2n-2k-1$.}
\end{multline*}
which generate the group $\CH(Y)$ modulo the norms by Lemma \ref{1}.
Note that
$$
2(n-2(k-1))-3\leq 2n-2k-1.
$$

Now let us check that Condition (b) holds.
Although the specialization homomorphism from the Chow group of the
generic fiber to the Chow group of the fiber over a point $x$ is not
surjective in general,
it is surjective by Lemma \ref{1} if the residue field of $x$ does not contain a subfield
isomorphic to $K$.
We finish the proof by showing that in the opposite case the image of
$\CH(Y_x)$ in the associated graded group of the filtration on $\CH(Y)$ is
in the image of $1+\sigma$.

Let $T$ be the closure of $x$ in $X$.
Let $Y_T=Y\times_XT\hookrightarrow Y$ be the preimage of $T$ under $Y\to
X$.
The image of the homomorphism $\CH(Y_x)\to \cF^r\CH(Y)/\cF^{r+1}\CH(Y)$,
where $r=\codim_X x$, is in the image of the push-forward
$\CH(Y_T)\to\cF^r\CH(Y)/\cF^{r+1}\CH(Y)$.
Since $x$ is the generic point of $T$ and $F(x)=F(T)\supset K$, a
non-empty open subset $U\subset T$ possesses a morphism to $\Spec K$.
Its preimage $Y_U\subset Y_T$ is open and also possesses a morphism to $\Spec
K$.
Therefore $(Y_U)_K\simeq Y_U\coprod Y_U$ and, in particular,
the push-forward $\CH((Y_U)_K)\to\CH(Y_U)$ is surjective.

We play with the following commutative diagram:
$$
\divide\dgARROWLENGTH by3
\begin{diagram}
\node{Y_K}\arrow[2]{e}\node{}\node{Y}\arrow[2]{e}\node{}\node{X}\\
\node{(Y_T)_K}\arrow{n,J}\arrow[2]{e}\node{}\node{Y_T}\arrow[2]{e}\arrow{n,J}\node{}
\node{T}\arrow{n,J}\\
\node{(Y_U)_K}\arrow{n,J}\arrow[2]{e}\node{}\node{Y_U}\arrow{n,J}\arrow[2]{e}\node{}\node{U}
\arrow{n,J}
\end{diagram}
$$

It follows that the image of the push-forward
$\CH((Y_T)_K)\to\CH(Y_T)$ generates $\CH(Y_T)$ modulo the image of
$\CH(Y_T\setminus Y_U)$.
Since the image of $\CH(Y_T\setminus Y_U)\to\CH(Y)$ is in
$\cF^{r+1}\CH(Y)$, it follows that
the image of $\CH(Y_T)$ in the quotient of the filtration on $\CH(Y)$ is
contained in the image of $\cF^r\CH(Y_K)$, that is, in the image of
$1+\sigma$.
\end{proof}

Let $I=[1,\;k]=\{1,2,\dots,k\}$.
For every $i\in I$,
let $\cA_i$ and $\cB_i$ the the tautological vector bundles on $H_i$
(so that $\cT_i=\cA_i\oplus\cB_i$).
We define $a_i,b_i\in\CH^1(H_I)$ as the first Chern classes of the line
bundles $(H_I\to H_i)^*\cA_i/(H_I\to H_{i-1})^*\cA_{i-1}$ and
$(H_I\to H_i)^*\cB_i/(H_I\to H_{i-1})^*\cB_{i-1}$.

For any $l\in I$, we identify $\CH(H_{[l\;,k]})$ with a subring in
$\CH(H_I)$ via the pull-back.
Note that $a_i,b_i\in\CH(H_{[l\;,k]})$ for $i\in[l+1\;,k]$.

By induction on $l\in I$, we prove the following statement;
note that this statement for $l=k$ is the statement of Proposition \ref{1.9}:

\begin{lemma}
\label{xxx}
The ring $\CH(H_{[l,\;k]})^\sigma$ is generated modulo $(1+\sigma)$ by
the elements of Proposition \ref{1.9} and the elements
$\{a_ib_i\}_{i\in[l+1,\;k]}$.
\end{lemma}

\begin{proof}
The induction base $l=1$ follows from Lemma \ref{1.10} and Lemma \ref{1.8}
(the latter showing that the missing generators of Lemma \ref{1.10} are
expressible in terms of the kept generators and the added generators).
Let us do the passage from $l-1$ to $l$.

The projection $H_{[l-1,\;k]}\to H_{[l,\;k]}$ is (canonically isomorphic to)
a product of two rank $l-1$ projective
bundles (given by the dual of the rank $l$ tautological vector bundles $\cA_{l}$ and $\cB_{l}$
on $H_{[l,\;k]}$).
The $\CH(H_{[l,\;k]})$-algebra $\CH(H_{[l-1,\;k]})$ is therefore generated by
the two elements $a_l,b_l$ subject to the two relations
$$
\Sum_{i=0}^{l}c_i(\cA_{l})a_l^{l-i}=0, \;\;
\Sum_{i=0}^{l}c_i(\cB_{l})b_l^{l-i}=0.
$$
In particular, the $\CH(H_{[l,\;k]})$-module $\CH(H_{[l-1,\;k]})$ is free, a basis is given by
the products $a_l^ib_l^j$ with $i,j\in[0,\;l-1]$.

The involution $\sigma$ exchanges $a$ and $b$.
Therefore the module $\CH(H_{[l-1,\;k]})^\sigma/(1+\sigma)$ over the ring
$\CH(H_{[l,\;k]})^\sigma/(1+\sigma)$
is free of rank $l$, a basis is given by
the (classes of the) products $a_l^ib_l^i$ with $i\in[0,\;l-1]$.
In particular, the $\CH(H_{[l,\;k]})^\sigma/(1+\sigma)$-algebra $\CH(H_{[l-1,\;k]})^\sigma/(1+\sigma)$
is generated by $a_lb_l$.
This generator satisfies the following equality in the quotient
$\CH(H_{[l-1,\;k]})^\sigma/(1+\sigma)$:
$$
\Sum_{i=0}^{l}c_{2i}(\cT_{l})(a_lb_l)^{l-2i}=0.
$$
This is the only relation on the generator because its powers up to $l-1$
form a basis.

Now let $C\subset\CH(H_{[l,\;k]})^\sigma/(1+\sigma)$ be the subring
generated by
the elements of Proposition \ref{1.9} and the elements
$\{a_ib_i\}_{i\in[l+1,\;k]}$.
Note that the coefficients of the above relation are in $C$:
they are expressible in terms of $c_i(-\cT_l)$ (which are non-zero modulo
$1+\sigma$ only for $i=0,2,\dots,n-2l$ by Lemmas \ref{1.6} and \ref{1}.
Therefore the subring of $\CH(H_{[l-1,\;k]})^\sigma/(1+\sigma)$ generated
by $C$ and $a_lb_l$ is also a free $C$-module of rank $l$.
On the other hand, this subring coincides with the total ring by the
induction hypothesis and it follows that
$C=\CH(H_{[l,\;k]})^\sigma/(1+\sigma)$.
This proved Lemma \ref{xxx}.
\end{proof}
Proposition \ref{1.9} is proved.
\end{proof}

\begin{rem}[{\bf Geometric description of the generators}]
Proposition \ref{1.9} provides us with generators of the ring
$\CH(Y_k)$ modulo the $K/F-$norms via the identification
$\CH(Y_k)=\CH(H_k)^\sigma$ of Example \ref{link}.
These generators have precisely the same geometric description as the
standard generators of the Chow ring of an orthogonal grassmannian.
Namely, they are obtained via the composition
$(Y_{\{1,k\}}\to Y_k)_*\compose(Y_{\{1,k\}}\to Y_1)^*$
out of the additive generators of $\CH(Y_1)$ modulo the norms.
Moreover, for any odd $i$ satisfying $n-2k+1\leq i\leq 2n-2k-1$, the generator
$c_i(-\cT_k)/2$ is the class of the Schubert subvariety of the subspaces
intersecting non-trivially a fixed totally isotropic subspace in $V$ of
certain $K$-dimension.
This is a consequence of Remark \ref{complete} and  Lemma \ref{1.6}.
\end{rem}

\section
{Steenrod operations for split unitary grassmannians}
\label{Steenrod operations for split unitary grassmannians}

In this section, dimension $n$ of the $K$-vector space $V$ is supposed to
be even.

Let $H=H_k$.
One more tool for study of $\CH(H)$ is given by the morphism
$\inc\colon H\to X$, where $X$ is the variety of totally isotropic
$2k$-planes of the hyperbolic quadratic form $\HH(V)=V\oplus V^*$.
The morphism associates to a point $(U,U')$ of $H$ the point
$U\oplus U'$ of $X$.
This is a closed imbedding by \cite[Corollary 10.4]{MR1758562}.

Note that the image of the pull-back $\inc^*\colon\CH(X)\to\CH(H)$
is contained in $\CH(H)^\sigma$.
Indeed, fixing
a non-degenerated symmetric bilinear form on $V$ giving
an identification of $V$ with $V^*$, we get the exchange
involution on $H$ (inducing $\sigma$ on $\CH(H)$) and an involution on $X$
given by the automorphism $V\oplus V^*=V^*\oplus V$.
The imbedding $H\hookrightarrow X$ commutes with these involutions, and the
involution induced on $\CH(X)$ is the identity because $V$ is of even
dimension.

The power of this tool
is explained by the fact that $\CH(X)$, in contrast to $\CH(H)$, is very well
studied.
An advantage of the variety $X$ is that (in contrast to $H$) it has twisted
forms with closed points of ``high'' degrees.

The meaning of the imbedding $H\hookrightarrow X$ is as follows.
Assume that $V$ is endowed with a $K/F$-hermitian form $h$.
We consider the variety $Y_k$.
Let $X_{2k}$ be the variety of $2k$-planes in the vector $F$-space $V$
totally isotropic with respect to the quadratic form on $V$ given by $h$.
We have a natural closed imbedding $\inc\colon Y_k\hookrightarrow X_{2k}$ which
becomes the above imbedding over $K$.
Choosing a hyperbolic $h$, we get another proof of the fact that the image
of $\CH(X)\to\CH(H)$ is in $\CH(H)^\sigma$: this is so because $\CH(X)=\CH(X_{2k})$ and
$\CH(Y_k)=\CH(H)^\sigma$.

Recall (see \cite{Vishik-u-invariant})
that the ring $\CH(X)$ is generated by certain elements
$w_i\in\CH^i(X)$, $i=0,1,\dots,n-2k$ and $z_i\in\CH^i(X)$,
$i=n-2k,n-2k+1,\dots,2n-2k-1$.
They satisfy $w_i=c_i(-\cT_X)$ for all $i$ and $z_i=c_i(-\cT_X)/2$ for $i\ne
n-2k$, where $\cT_X$ is the tautological vector bundle on $X$.

\begin{lemma}
\label{in*}
The pull-back $\CH(X)\to \CH(H)^\sigma/(1+\sigma)$ is surjective.
The image of each $z_i$ with even $i\ne n-2k$ is $0$.
\end{lemma}

\begin{rem}
One may show (see \cite{gug}) that the pull-back
$$
\inc^*\colon \CH(X_{2k})\to \CH(Y_k)/(1+\sigma)
$$
is surjective (for any $h$).
Moreover, the push-forward $\inc_*$ induces an injection
$$
\inc_*\colon\CH(Y_k)/(1+\sigma)\to\Ch(X_{2k}),
$$
where $\Ch:=\CH/2$.
It follows that the ring $\CH(Y_k)/(1+\sigma)$ is naturally identified with
$\Ch(X_{2k})$ modulo the kernel of the multiplication by
$[Y_k]\in\Ch(X_{2k})$.
In the case of $k=n/2$ and hyperbolic $h$, the computation of the class $[Y_k]$ given below
together with the computation of $\Ch(X_{2k})$ given in
\cite{EKM},  provides the following presentation of $\CH(Y_{n/2})/(1+\sigma)$
by generators and relations:
generators are $e_i\in\CH^i$, $i=1,3,\dots,n-1$;
relations are $e_i^2=0$ for each $i$.
\end{rem}

\begin{proof}[Proof of Lemma \ref{in*}]
The generators of the ring on the right-hand side given in Proposition \ref{1.9}
come from $\CH(X)$
because the pull-back of the tautological vector bundle $\cT_X$ on $X$ to
$H$ is $\cT_k$ and the Chern classes $c_i(-\cT_X)$ are divisible by $2$
for $i>n-2k$.
This gives the surjectivity.

The image of $z_i$ with even $i\ne n-2k$ is $0$ by Lemmas \ref{1.6} and
\ref{1}.
\end{proof}

\begin{lemma}
\label{[Y]}
The element $[H]\in\CH(X)$ is a square.
\end{lemma}

\begin{proof}
Let $x\in\CH(X)$ be the class of the Schubert subvariety $S\subset X$ of the
subspaces $U\subset V\oplus V^*$ satisfying $\dim U\cap V\geq k$.
We claim that $[H]=x^2$.
Indeed, $x$ can be also represented by the Schubert subvariety $S'\subset X$
of the
subspaces $U\subset V\oplus V^*$ satisfying $\dim U\cap V^*\geq k$.
Since $S\cap S'=H$ and $\codim_XH=\codim_XS+\codim_XS'$, $[H]=[S]\cdot[S']$
by \cite[Corollary 57.22]{EKM}.
\end{proof}

We now pass to the modulo $2$ Chow group $\Ch(X)=\CH(X)/2$ and
we use the notion of {\em level} for elements of $\Ch(X)$ introduced in
\cite{oddisotro-tignol}.
Namely, an element of $\Ch(X)$ is of level $l$ if it can be written as a
polynomial in the generators of the $z$-degree $\leq l$
(we use the same notation $w_i$, $z_i$ for the classes of the integral generators).
We recall that (see \cite[proof of Proposition 12]{oddisotro-tignol}) by the
formula of \cite[Proposition 2.9]{Vishik-u-invariant} the cohomological
Steenrod operation preserves the level.
In particular, the squaring preserves the level.

We also recall that the generators satisfy the relation
$$
z_i^2=z_ic_i(-\cT_X)-z_{i+1}c_{i-1}(-\cT_X)+z_{i+2}c_{i-2}(-\cT_X)-\dots
$$
which shows that any element of $\Ch(X)$ can be written as a polynomial in
the generators of $z_i$-degree $\leq1$ for each $i$.
A polynomial satisfying this restriction is called {\em standard}.

\begin{cor}
\label{2.4}
The element $[H]\in\Ch(X)$ is of level $k$.
\end{cor}

\begin{proof}
Since squaring does not affect the level, it suffices to show that the
level of a homogeneous element $x$ with $x^2=[H]$ is $k$.
The codimension of $x$ is equal to
\begin{multline*}
(\dim X-\dim H)/2=\big(k(4n-6k-1)-k(2n-3k)\big)/2=
(k/2)(2n-3k-1)=\\
(n-2k)+(n-2k+1)+\dots+(n-k-1),
\end{multline*}
and the minimal codimension of an element which is not of level $k$ is
this number plus $n-k$.
\end{proof}

\begin{thm}[{cf. \cite[Proposition 12]{oddisotro-tignol}}]
\label{main Steen}
Let $F$ be a field of characteristic $\ne2$,
$K/F$ a quadratic field extension, $V$ a vector space
over $K$ of even positive dimension $n$, $h$ a $K/F$-hermitian form on
$V$, $k$ an integer satisfying $1\leq k\leq n/2$,
$Y$ the variety of totally isotropic $k$-planes in $V$.
Then for any $i>k(n-2k)$, one has $\deg\Steen\Ch_i(Y_k)=0$,
where $\Steen$ is the cohomological Steenrod operation and $\deg$ is the degree homomorphism on the
modulo $2$ Chow groups.
\end{thm}

\begin{proof}
Assume that $\deg\Steen\Ch^j(Y)\ne0$ for some $j$.
Then $\deg\Steen\Ch^j(H)^\sigma\ne0$.
Since $\Steen$ commutes with $\sigma$, $\Steen$ is trivial on
$(1+\sigma)$.
Therefore $\deg\Steen\Ch^j(H)^\sigma/(1+\sigma)\ne0$.
It follows by Lemma \ref{in*} that $\deg\inc^*\Steen\Ch^j(X)\ne0$,
or, equivalently, $\deg\Steen\inc^*\Ch^j(X)\ne0$.
Let $y\in\Ch^j(X)$ be a standard monomial in the generators with $\deg\Steen\inc^*(y)\ne0$.
Since $\inc^*(y)\ne0$, the monomial $y$ does not contain any
$z_i$ with even $i\ne n=2k$ by the second half of Lemma \ref{in*}.
We may also assume that $y$ does not contain $z_{n-2k}$.
Indeed, $\inc^*(z_{n-2k})$ is a polynomial in the generators of $\Ch(H)^\sigma/(1+\sigma)$
of codimension $\leq n-2k$.
In particular, $\inc^*(z_{n-2k})$ is a polynomial in $c_i(-\cT_k)$ with $i\leq
n-2k$.
Let $P\in\Ch^{n-2k}(X)$ be the same polynomial in $c_i(-\cT_X)=w_i$.
Then $\inc^*(P)=\inc^*(z_{n-2k})$ and we may replace $z_{n-2k}$ by $P$ in
$y$ without changing $\inc^*(y)$.

We have
$$
0\ne\deg\inc^*\Steen(y)=\deg\inc_*\inc^*\Steen(y)=
\deg\Steen([H]\cdot y).
$$
Since degree of any level $2k-1$ element is $0$,
\cite[proof of Proposition 12]{oddisotro-tignol}, the element
$\Steen([H]y)$ is not of level $2k-1$.
Since the Steenrod operation preserves the level,
the product $[H]\cdot y$ is not of level $2k-1$.
Since $[H]$ is of level $k$ by Corollary \ref{2.4}, $y$ is not of level
$k-1$.
The smallest possible codimension of a monomial of level not $k-1$
without $z$-generators of even codimension is the sum of $k$ summands
$$
(n-2k+1)+(n-2k+3)+\dots+(n-1)=k(n-2k+1)+k(k-1)=k(n-k).
$$
It follows that $j<k(n-k)$.
Since $\dim Y-k(n-k)=k(n-2k)$, we are done.
\end{proof}

\section
{Some ranks of some motives}
\label{Some ranks of some motives}

Let $K/F$ be a separable quadratic field extension.
Let $M$ be a motive over $F$ with coefficients in $\F_2$.
We assume that there
exists a field extension $F'/F$ linearly disjoint with an algebraic closure of $F$
such that the motive $M_{F'}$ decomposes in a sum of shifts of the motives
of $\Spec F'$ and $\Spec K'$, where $K'$ is the field $K\otimes_FF'$.
Note that the number of $F'$ and the number of $K'$ appearing in the
decomposition do not depend on the choice of $F'$:
if $F''$ is another field like that, the Krull-Schmidt principle
\cite{MR2110630} over the
field of fractions of $F'\otimes_F F''$ gives the equalities.
Here we use an easy version of the Krull-Schimdt principle for motives
with finite coefficients of quasi-homogeneous varieties proved also in
\cite[Corollary 2.2]{outer}.

The number of $F'$ in the above decomposition is the {\em $F$-rank} $\rk_F$ of
$M$, the number of $K'$ is the {\em $K$-rank} $\rk_K$ of
$M$.
The usual rank $\rk M$ is also defined for such $M$ and is equal to
$\rk_FM+2\rk_KM$.

Recall that there are functors
$$
\tr,\cores\colon\CM(K,\F_2)\to\CM(F,\F_2).
$$
The first one (non-additive and not commuting with the shift, see \cite{MR1809664})
is induced by the Weil transfer.
The second one (additive and commuting with the shift, see \cite{outer}) is induced
by the functor associating to a $K$-variety the same variety considered as
a variety over $F$ via the composition with $\Spec K\to\Spec F$.

Here is an example of computation of ranks.

\begin{lemma}
\label{primer}
Let $M$ be a motive over $K$ isomorphic to a sum of $n$ shifts of the Tate
motive.
Then $\rk_F\tr M=\rk M=n$, $\rk_K\tr M=n(n-1)/2$,
$\rk_F\cores M=0$, and $\rk_K\cores M=n$.
\end{lemma}

\begin{proof}
Since $\cores M(\Spec K)=M(\Spec K)$,
the formulas for $\cores$ follow.
The formulas for $\tr$ follow from \cite[Lemma 2.1]{qweil}.
\end{proof}

Let $D$ be a central division $K$-algebra admitting a $K/F$-unitary
involution, and assume that $\deg D=2^n$ for some $n\geq0$.
For an integer $k\in[0,\;n-1]$, let $X_k$ be the Weil transfer with respect
to $K/F$ of the generalized Severi-Brauer variety $\SB(2^k,D)$.
The motive $M(X_k)$ satisfies the above conditions
(one may take as $F'$ the function field of the variety $X_0$)
so that the ranks
$\rk_FM$ and $\rk_KM$ are defined for any summand $M$ of $M(X_k)$.
In particular, the ranks
$\rk_FU(X_k)$ and $\rk_KU(X_k)$ are defined for the {\em upper} (indecomposable)
motive $U(X_k)$.
As in \cite{canondim},
$U(X_k)$ is defined as the summand in the complete motivic
decomposition of $X_k$ satisfying the condition $\Ch^0(U(X_k))\ne0$.

\begin{prop}
\label{ranks}
$v_2(\rk_FU(X_k))=n-k$, $v_2(\rk_KU(X_k))=n-k-1$.
\end{prop}

\begin{proof}
The proof proceeds by induction on $k$.
Let us do the induction base $k=0$.
According to \cite[Theorem 1.2]{qweil}, $U(X_0)=M(X_0)$.
Since $\rk M(\SB(1,D))=2^n$,
it follows from Lemma \ref{primer} that
$\rk_F U(X_0)=2^n$, $\rk_K U(X_0)=2^{n-1}(2^n-1)$.

Now we assume that $k>0$.
Since
$\rk M(\SB(2^k,D))=b:=\binom{2^n}{2^k}$,
it follows from Lemma \ref{primer} that
$\rk_FM(X_k)=b$ and $\rk_KM(X_k)=b(b-1)/2$.
In particular, $v_2(\rk_FM(X_k))=n-k>0$ and $v_2(\rk_KM(X_k))=n-k-1$.
Therefore, it suffices to show that for each summand $M$ different from $U(X_k)$
in the complete motivic decomposition of $X_k$, we have $v_2(\rk_FM)>n-k$
and $v_2(\rk_KM)>n-k-1$.

By \cite{outer} and \cite{qweil}, $M$ is a shift of the motive $U(X_l)$ with some
$l\in[0,\;k-1]$ or a shift of the motive $\cores_{K/F}U(\SB(2^l,D))$ with
some $l\in[0,\;k]$.
In the first case we are done by the induction hypothesis.
In the second case we have $\rk_FM=0$ and $\rk_KM=\binom{2^n}{2^l}$.
\end{proof}

\section
{Unitary isotropy theorem}
\label{Unitary isotropy theorem}

Let $K$ be a field, $A$ a central simple $K$-algebra,
$\tau$ a unitary involution on $A$,
$F$ the subfield of the elements of $K$ fixed under $\tau$.
We say that $\tau$ is {\em isotropic}, if $\tau(I)\cdot I=0$ for some
non-zero right ideal $I\subset A$;
otherwise we say that $\tau$ is {\em anisotropic}.

\begin{thm}[{\bf Unitary Isotropy Theorem}]
\label{main}
Assume that $\Char F\ne2$.
If $\tau$ becomes isotropic over any field extension $F'/F$  such that $K':=K\otimes_FF'$ is a field
and the central simple $K'$-algebra $A':=A\otimes_FF'$ is split,
then $\tau$ becomes isotropic over some finite odd degree field extension of $F$.
\end{thm}

\begin{proof}

We can easily reduce this theorem to the case of $2$-primary $\ind A$.
Indeed, it suffices to find a finite odd degree field extension $L/F$,
such that $A$ becomes $2$-primary over $L$.
For such $L/F$ we can take the field extension of $F$ corresponding to a Sylow $2$-subgroup
of the Galois groups of the normal closure of $E/F$, where $E$ is a separable
finite odd degree field extension of $K$ such that $\ind (A\otimes_F E)$ is $2$-primary.

Because of the above reduction, we assume that the index of $A$ is a power of $2$.

We follow the lines of the proof of \cite[Theorem 1]{oddisotro-tignol}.
We prove Theorem \ref{main} over all fields simultaneously using an induction on $\ind A$.
The case of $\ind A =1$ is trivial.
From now we are assuming that $\ind A =2^r$ for some integer $r\geq 1$, and we fix the following notations:

$F$ is a field of characteristic different from $2$;

$K/F$ is a quadratic field extension;

$A$ is a central simple $K$-algebra of the index $2^r$ (with $r\geq1$);

$\tau$ is an $F$-linear unitary involution on $A$;

$D$ is a central division $F$-algebra (of degree $2^r$) Brauer-equivalent to $A$;

$V$ is a right $D$-module of $D$-dimension $v$ with an isomorphism $\End_D (V) \simeq A$
(in particular, $\rdim V = \deg A = 2^r \cdot v$, where $\rdim V:=\dim_FV/\deg D$ is the reduced dimension);

we fix an arbitrary $F$-linear unitary involution $\varepsilon$ on $D$;

$h$ is a hermitian (with respect to $\varepsilon$) form on $V$ such that the involution $\tau$ is adjoint to $h$;

$Y = X(2^r;(V,h)) \simeq X(2^r;(A,\tau))$ is the variety of totally isotropic submodules in $V$ of reduced
dimension $\rdim = 2^r$ which is isomorphic (via Morita equivalence)
to the variety of right totally isotropic ideals in $A$ of the same reduced dimension;

$X$ is the Weil transfer (with respect to $K/F$) of the generalized Severi-Brauer $K$-variety $X(2^{r-1};D)$.

We are going to apply the assumption of Theorem \ref{main} to only one
field extension $F'/F$, namely, to
the function field of the Weil transfer
of the Severi-Brauer variety $X(1;D)$ of $D$.
So, starting from this point, $F'$ stands for this function field.
Clearly, $K':=K\otimes_FF'$ is a field and $A'$ is split.
We assume that the involution $\tau'$ (and therefore, the hermitian form $h_{F'}$)
is isotropic  and we want to show that $h$ (and $\tau$) becomes isotropic over a finite
odd degree extension of $F$.
According to \cite[Theorem 1.4]{unitary}, the Witt index of $h_{F'}$ is a multiple of $2^r = \ind A$.
In particular, $v\geq2$.
If the Witt index is greater than $2^r$, we replace $V$ by a submodule in $V$ of $D$-codimension $1$
and we replace $h$ by its restriction on this new $V$.
The Witt index of $h_{F'}$ drops by $2^r$ or stays unchanged.
We repeat the procedure until the Witt index becomes equal to $2^r$.
In particular, $v$ is still $\geq2$.

The variety $Y$ has an $F'$-point and the index of the central simple
$K\otimes_FF(X)$-algebra $A\otimes_FF(X)$ is equal to $2^{r-1}$
(note that $K\otimes_FF(X)$ is a field).
Consequently, by the induction hypothesis, the variety $Y_{F(X)}$ has an odd degree closed point.
We prove Theorem \ref{main} by showing that the variety $Y$ has an odd degree closed point.

We will use and we recall the following statement from \cite{oddisotro-tignol}.

\begin{prop}
\label{summand}
Let $\mathfrak{X}$ be a geometrically split, geometrically irreducible $F$-variety satisfying the
nilpotence principle and let $\mathcal{Y}$ be a smooth complete $F$-variety.
Assume that there exists a field extension $E/F$ such that
\begin{enumerate}
  \item for some field extension $\overline{E(\mathfrak{X})}/E(\mathfrak{X})$, the image of the change
  of field homomorphism $\Ch(\mathcal{Y}_{E(\mathfrak{X})}) \rightarrow \Ch(\mathcal{Y}_{\overline{E(\mathfrak{X})}})$
  coincides with the image of the change of field homomorphism
  $\Ch(\mathcal{Y}_{F(\mathfrak{X})}) \rightarrow \Ch(\mathcal{Y}_{\overline{E(\mathfrak{X})}})$;
  \item the $E$-variety $\mathfrak{X}_E$ is $p$-incompressible;
  \item a shift of the upper indecomposable summand of $M(\mathfrak{X})_E$ is a summand of $M(\mathcal{Y})_{E}$.
\end{enumerate}

Then the same shift of the upper indecomposable summand of $M(\mathfrak{X})$ is a summand of $M(\mathcal{Y})$.
\end{prop}

We are going to apply Proposition \ref{summand} (with $p=2$) $\mathfrak{X} =X$, $\mathcal{Y} =Y$,
and $E=F(Y)$.
We need to check that conditions (1) - (3) are satisfied for these $X,Y,E$.
First of all, we need a motivic  decomposition of $Y$ over a field extension $\tilde{F}/F$,
such that $Y(\tilde{F}) \neq \emptyset$ and $\tilde{K}=K\otimes_F\tilde{F}$ is a field.
Over such $\tilde{F}$, the hermitian form $h$ decomposes in the orthogonal sum of the hyperbolic
$\tilde{D}$-plane and a hermitian form $h'$ on a right $\tilde{D}$-module $V'$ with $\rdim V'=2^r(v-2)$,
where $\tilde{D}$  is central simple $\tilde{K}$-algebra $D\otimes_F\tilde{F}$.
Let $L/F(X)$ be a finite odd degree extension such that $Y(L)\neq \emptyset$.
Recall that a smooth projective variety is {\em anisotropic}, if it has no odd degree
closed points (by \cite[lemma 6.3]{hypernew-tignol}, the motive of an anisotropic variety does not contain a Tate
summand).

\begin{lemma}
\label{mdec}
The shift of the motive of $X_{\tilde{F}}$ and two Tate motives are the motivic summands of $Y_{\tilde{F}}$.
In the case $\tilde{F} = L$, any other motivic summand of $Y_L$ is a shift of some anisotropic $L$-variety.
\end{lemma}

\begin{proof}
According to \cite[Theorem 15.8]{MR1758562}, the variety $Y_{\tilde{F}}$ is a relative cellular space
(as defined in \cite[\S66]{EKM}) over the (non-connected) variety $Z$ of triples $(I,J,N)$,
where $I$ and $J$ are right ideals in $D$, and where $N$ is a submodule in $V'$ such that the
submodule $I \oplus J \oplus N \subset V$ is a point of $Y_{\tilde{F}}$
(that is, $\varepsilon_{\tilde{F}}(I) \cdot J = 0$, $N$ is totally isotropic, and the reduced
dimension of the $\tilde{D}$-module $I \oplus J \oplus N \subset V$ is equal to $\deg \tilde{D}$).
Therefore, by \cite[Corollary 66.4]{EKM},
the motive of $Y_{\tilde{F}}$ is the sum of shifts of the motives of the components of $Z$.

The shift of the motive of $X_{\tilde{F}}$ is given by the motive of the component of the triples
$\{(I,J,0) | \rdim I = \rdim J = (\deg \tilde{D})/2 \}$.
The rational points $(0,\tilde{D},0)$ and $(\tilde{D},0,0)$ of $Z$ are components of $Z$
which produce the two promised Tate summands.
In the case $\tilde{F} = L$ we have $\ind \tilde{D} = (\deg \tilde{D})/2 = 2^{r-1}$.
Therefore to prove the second statement of this lemma, we only need to check that the component of
$Z$ of triples $(0, 0, N)$ is anisotropic.
It is true, because this component is naturally identified with anisotropic $L$-variety $Y' = X(2^r; (V',h'))$.
\end{proof}

\begin{rem}
\label{rem}
Two Tate motives mentioned in Lemma \ref{mdec} are clearly $\F_2$ and $\F_2(\dim Y)$.
In the case $\tilde{F} = L$, by duality, the motivic summand $M(X_L)$ of $Y_L$ has as the shifting number the integer
$$d:= (\dim Y - \dim X)/2 \, .$$
\end{rem}

Since $Y(F(Y))\neq \emptyset$, the condition (3) of Proposition \ref{summand} is checked by Lemma \ref{mdec}.
Let us check now the condition (2).
By \cite[Theorem 1.1]{qweil}, the variety $X_{F(Y)}$ is $2$-incompressible if (and only if)
the $K\otimes_FF(Y)$-algebra $D\otimes_FF(Y)$ is division.
This is indeed the case:

\begin{lemma}
\label{div}
The algebra $D\otimes_FF(Y)$ is division, that is, $\ind (D\otimes_FF(Y)) = \ind D$.
\end{lemma}

\begin{proof}
The proof is similar to the proof of \cite[Lemma 6]{oddisotro-tignol}. Assume that $\ind (D\otimes_FF(Y)) < \ind D$.
Then we could prove as in \cite[Lemma 6]{oddisotro-tignol}, that the upper indecomposable motivic summand of $X$
is a motivic summand of $Y$.
This implies (because the variety $X$ is $2$-incompressible) that the complete motivic decomposition of the
variety $Y_{F(X)}$ contains the Tate summand $\F_2(\dim X)$.
By Lemma \ref{mdec} and Remark \ref{rem} we get a contradiction.
\end{proof}

We have checked condition (2) of Proposition \ref{summand}.
To check the remaining condition (1), we will need the same property for the variety $Y$ as in
\cite[Lemma 7]{oddisotro-tignol}.
We can prove it for more general class of varieties.
Let $\mathcal{Z}$ be a projective homogeneous variety under an arbitrary absolutely simple
algebraic group $G$ of type
$A_n$ over a field $k$
(we can replace ``absolutely simple of type $A_n$'' by the condition, that $G$ is semisimple and becomes of inner type over
some quadratic separable field extension of $k$).
In other words, $\mathcal{Z}$ is a variety of flags of isotropic right ideals of a central
simple algebra over a quadratic separable field extension of $k$ endowed (the algebra)
with a unitary $k$-linear involution.

\begin{lemma}
\label{Gal}
Let $k'/k$ be a finite odd degree field extension and let
$\bar{k}$ be an algebraic closure of $k$ containing $k'$.
Then $\Im(\Ch({\mathcal{Z}})\to\Ch({{\mathcal{Z}}_{\bar{k}}}))=\Im(\Ch({\mathcal{Z}}_{k'})\to\Ch({{\mathcal{Z}}_{\bar{k}}}))$.
\end{lemma}

\begin{proof}
For any field extension $E\subset \bar{k} $ of $k$, we write $I_E$
for the image of $\Ch({\mathcal{Z}}_E)\to\Ch({{\mathcal{Z}}_{\bar{k}}})$.
We only need to show that $I_{k'}\subset I_k$ because,
clearly, $I_k\subset I_{k'}$.

If $G$ is of inner type, the variety ${\mathcal{Z}}$ is a variety of flags of right ideals of a central simple
$k$-algebra.
Therefore the group $\Aut(\bar{k}/k)$ acts trivially
on $\Ch({{\mathcal{Z}}_{\bar{k}}})$.
It follows that $[k':k]\cdot I_{k'}\subset I_k$ and therefore $I_{k'}\subset
I_k$.

Now we assume that $G$ is of outer type.
Let $K\subset \bar{k}$ be the separable quadratic field extension
of $k$ such that $G_K$ is of inner type. Consider two subgroups
$\Aut(\bar{k}/K)$ and $\Aut(\bar{k}/k')$ of the group $\Aut(\bar{k}/k)$. Acting on $\Ch({{\mathcal{Z}}_{\bar{k}}})$, they act trivially on $I_{k'}$.
The index of the first subgroup is $2$ while the index of the second one
is odd (a divisor of $[k':k]$).
Indeed,
$$
\Aut(\bar{k}/k)=\Aut(k_\mathrm{sep}/k),\;\; \Aut(\bar{k}/K)=\Aut(k_\mathrm{sep}/K),
$$
where $k_\mathrm{sep}$ is the separable closure of $k$ in
$\bar{k}$,
so that $\Aut(\bar{k}/k)/\Aut(\bar{k}/K)=\Aut(K/k)$;
if $k''$ is the separable closure of $k$ in $k'$, then $\Aut(\bar{k}/k')=\Aut(\bar{k}/k'')$, so that
the index of $\Aut(\bar{k}/k')$ in $\Aut(\bar{k}/k)$ is $[k'':k]$.

It follows that $\Aut(\bar{k}/k)$ acts trivially on $I_{k'}$.
Therefore we still have the inclusion $[k':k]\cdot I_{k'}\subset I_k$ giving $I_{k'}\subset I_k$.
\end{proof}

\begin{cor}
\label{cor}
$U(X)(d)$ is a motivic summand of $Y$.
\end{cor}
\begin{proof}
As planned, we apply Proposition \ref{summand} to $p=2$, $\mathfrak{X} =X$, $\mathcal{Y} =Y$, and $E=F(Y)$.
Since $E(X)\subset L(Y)$, we have the commutative diagram
$$
\begin{CD}
\CH(Y_{E(X)}) @>>> \CH(Y_{L(Y)}) @>>> \CH(Y_{\overline{L(Y)}})\\
@AAA  @AAA \\
\CH(Y_{F(X)}) @>>> \CH(Y_L)
\end{CD}
$$
where the maps are the change of field homomorphisms and where $\overline{L(Y)}$ is an algebraic closure of $L(Y)$.
We check condition (1) for $\overline{E(\mathfrak{X})} = \overline{L(Y)}$. For any field extension $\mathcal{F}\subset \overline{L(Y)} $ of $F$, we write $I_{\mathcal{F}}$
for the image of $\Ch({\mathcal{Y}}_{\mathcal{F}})\to \Ch({\mathcal{Y}}_{\overline{L(Y)}})$. We only need to show that $I_{E(X)}\subset I_{F(X)}$. We have $I_{E(X)}\subset I_{L(Y)}$. Since $Y(L) \neq \emptyset$, the field extension $L(Y)/L$ is purely transcendental. Therefore ${\res}_{L(Y)/L}$ is surjective and $I_{L(Y)} = I_L$. Finally, by Lemma \ref{Gal}, $I_L = I_{F(X)}$. We obtain the necessary inclusion $I_{E(X)}\subset I_{L(Y)} = I_L = I_{F(X)}$.

As already pointed out, condition (2) is satisfied by Lemma \ref{div}, and condition (3) is satisfied by Lemma \ref{mdec}. Therefore, by Proposition \ref{summand}, a shift of $U(X)$ is a motivic summand of $Y$.
By Remark \ref{rem}, it follows that the shifting number of this motivic summand $U(X)$ is equal to $d$.
\end{proof}

As in \cite{oddisotro-tignol} we need the following enhancement of Corollary \ref{cor}.

\begin{prop}
There exists a \underline{symmetric} projector $\pi$ on $Y$
such that the motive $(Y,\pi)$ is isomorphic to $U(X)(d)$.
\end{prop}

\begin{proof}
We can follow the lines of the proof of \cite[Proposition 9]{oddisotro-tignol}
if we know that the complete motivic decomposition of $Y_{F(X)}$ could not contain two copies of $\F_2(d)$.
This is true by Lemma \ref{mdec} and Remark \ref{rem}.
\end{proof}

The following proposition finishes the proof of Theorem \ref{main}.

\begin{prop}
Let $F$ be a field of characteristic $\ne2$.
Let $K/F$ be a
quadratic field extension.
Let $D$ be a central division $K$ algebra of degree $2^r$ with some
$r\geq1$ admitting a $K/F$-unitary involution.
Let $X$ be the Weil transfer of the generalized Severi-Brauer variety
$\SB(2^{r-1},D)$.
Let $A$ be a central simple $K$-algebra Brauer-equivalent to $D$ endowed with a
$K/F$-unitary involution.
Let $Y$ be the variety of isotropic rank $2^r$ right ideals in $A$.
Assume that there is a symmetric projector $\pi\in\Ch_{\dim Y}(Y\times Y)$
such that the motive $(Y,\pi)$ is isomorphic to $U(X)(d)$.
Then $Y$ has a closed point of odd degree.
\end{prop}

\begin{proof}
By Lemma \ref{symmop}, it is enough to show that $\sq(\pi)\ne\st(\pi)$.
Computing $\sq$ and $\st$, we may go over any field extension of $F$.
There exists a field extension $\tilde{F}/F$
over which $A$ is split
and the unitary involution on $A$ is hyperbolic,
but $\tilde{K}:=K\otimes_F\tilde{F}$ is still a field.
The variety $Y_{\tilde{F}}$ can be identified with the variety of
$2^r$-dimensional totally isotropic subspaces of some vector space
$V/\tilde{F}$ endowed with a hyperbolic $\tilde{K}/\tilde{F}$-hermitian form $h$.

Since the motive of $X$ over $\tilde{F}$ is a sum of shifts of the motives
of $\Spec \tilde{F}$ and $\Spec \tilde{K}$, $\pi$ decomposes in a sum of
two orthogonal projectors $\alpha$ and $\beta$ such that $(Y_{\tilde{F}},\alpha)$ is a
sum of shifts of the motive of $\Spec\tilde{F}$ and $(Y_{\tilde{F}},\beta)$
is a sum of shifts of the motive of $\Spec\tilde{K}$.

First of all we will show that the projectors $\alpha$ and $\beta$ can be
chosen to be symmetric.
Let us consider the $\F_2$-vector space $\Ch(Y_{\tilde{K}})$ together with
the non-degenerate symmetric bilinear form $b\colon(v,u)\mapsto\deg(v\cdot
u)\in\F_2$.
Since $\pi$ is symmetric, the image $V:=\Ch_*(Y_{\tilde{K}},\pi_{\tilde{K}})$
of the projector
$\pi_*\colon\Ch(Y_{\tilde{K}})\to\Ch(Y_{\tilde{K}})$ is
orthogonal to its kernel.
In particular, the subspace $V$ is non-degenerate (with respect to $b$).
Since the projectors
$\alpha_*,\alpha^*=(\alpha^t)_*\colon V\to V$ are adjoint,
we have $(\Im\alpha^*)^\bot=\Ker\alpha_*$.
The subspace $\Im\alpha_*\subset V$ is non-degenerate:
since $(1+\sigma)V\subset\rad V^\sigma$ (because $b$ is $\sigma$-invariant) and
$$
\Im\alpha_*\oplus(1+\sigma)V=V^\sigma=\Im\alpha^*\oplus(1+\sigma)V,
$$
we have that
$
\rad\Im\alpha_*\subset
(\Im\alpha_*)\cap(\Im\alpha^*)^\bot=
(\Im\alpha_*)\cap(\Ker\alpha_*)=0.
$
It follows that $V=\Im\alpha_*\oplus(\Im\alpha_*)^\bot$.
Since the subspace $(\Im\alpha_*)^\bot$ is homogeneous and $\sigma$-invariant, the
orthogonal projections of $V$ onto the summands of this orthogonal
decomposition are realized by some (uniquely determined) projectors
$$
\alpha',\beta'\in\End(Y_{\tilde{F}},\pi_{\tilde{F}})\subset\Ch_{\dim Y}(Y\times Y)_{\tilde{F}}.
$$
The projectors $\alpha'$ and $\beta'$ are orthogonal, symmetric, and satisfy $\alpha'+\beta'=\pi$.
Since the motive $(Y_{\tilde{F}},\alpha')$ is split of the same rank as
$(Y_{\tilde{F}},\alpha)$,
the motive $(Y_{\tilde{F}},\beta')$ is a sum of shifts of the motive of $\Spec
K$ by the Krull-Schmidt principle.

Replacing $\alpha$ by $\alpha'$ and $\beta$ by $\beta'$,
we have (see Lemmas \ref{new1.2} and \ref{1.5}):
$\sq(\alpha)=\rk_F(Y,\pi)\pmod{4}=\rk_FU(X)\pmod{4}=2$ and
$\sq(\beta)=2\rk_K(Y,\pi)\pmod{4}=2\rk_KU(X)\pmod{4}=2$
by Proposition \ref{ranks}.
On the other hand, $\st(\alpha)=0$.
Indeed, $\alpha$ over $\tilde{F}$ is a sum of $a\times b$ with $a,b\in\Ch_{\geq
d}(Y_{\tilde{F}})$, where
$d=(\dim Y-\dim X)/2=\big(k(2n-3k)-k^2/2\big)/2$ with $k:=2^r=\ind A$ and $n:=\deg A$.
Since $d>k(n-2k)$, $\deg\Steen(a)=0$ by Theorem \ref{main Steen}.
Therefore $\pr_{2*}\Steen(a\times b)=0$.
It follows that $\pr_{2*}(\mf{a})$ is divisible by $2$ for an integral representative $\mf{a}$
of $\Steen(a\times b)$.
Therefore $\pr_{2*}(\mf{a})$ is divisible by $2$ if now $\mf{a}$ is an integral
representative of $\Steen(\alpha)$.
It follows that $\pr_{2*}(\mf{a})^2$ is divisible by $4$ and consequently
$\st(\alpha)=0$.

Finally, let us check that $\st(\beta)=\sq(\beta)$.
The point is that $\beta_{\tilde{K}}$ is in
$$
(1+\sigma)\Ch(Y\times Y)_{\tilde{K}}=
\Im\big(\Ch(Y\times Y)_{\tilde{K}}\to
\Ch(Y\times Y)_{\tilde{F}}\to\Ch(Y\times Y)_{\tilde{K}}\big).
$$
Therefore the element $\beta_{\tilde{K}}\in\Ch(Y\times Y)_{\tilde{K}}$
is a sort of ``always rational'' element:
this is an element of the Chow group of the square of the {\em completely split} unitary grassmannian such that
for {\em any} hermitian form $h'$ (hyperbolic or not and over any field $F'$) of the same dimension and the corresponding unitary grassmannian
$Y'$, this element considered in $\beta\in\Ch(Y'\times Y')_{\bar{F'}}=\Ch(Y\times Y)_{\tilde{K}}$ is rational.
Taking {\em anisotropic} $h'$
(in which case the variety $Y'$ has no odd degree closed points), we get
by Lemma \ref{symmop} that $\st(\beta)=\sq(\beta)$.

We have
calculated the values of the operations $\sq$ and $\st$ on $\alpha$ and
$\beta$.
We have by Lemmas \ref{new1.2} and \ref{new st}
that $\sq(\pi)=\sq(\alpha)+\sq(\beta)=0$ and
$\st(\pi)=\st(\alpha)+\st(\beta)=2$.
In particular, $\sq(\pi)\ne\st(\pi)$.
\end{proof}

Theorem \ref{main} is proved.
\end{proof}


\begin{thebibliography}{10}

\bibitem{MR2110630}
Chernousov, V., Gille, S. and Merkurjev, A.,
Motivic decomposition of isotropic projective homogeneous varieties.
Duke Math. J. 126, 1 (2005), 137--159.

\bibitem{EKM}
Elman, R., Karpenko, N. and Merkurjev, A.,
The Algebraic and Geometric Theory of Quadratic Forms, vol.~56
of American Mathematical Society Colloquium Publications.
American Mathematical Society, Providence, RI, 2008.

\bibitem{Fulton}
Fulton, W.,
Intersection Theory, second~ed., vol.~2 of Ergebnisse der
Mathematik und ihrer Grenzgebiete. 3. Folge. A Series of Modern Surveys in
Mathematics [Results in Mathematics and Related Areas. 3rd Series. A Series
of Modern Surveys in Mathematics].
Springer-Verlag, Berlin, 1998.

\bibitem{oddisotro-tignol}
Karpenko, N.~A.,
Isotropy of orthogonal involutions.
With an appendix by J.-P. Tignol. arXiv:0911.4170v3 [math.AG] (31 Jan
2010), 13 pages. Amer. J. Math., to appear.

\bibitem{MR1758562}
Karpenko, N.~A.,
Cohomology of relative cellular spaces and of isotropic flag
varieties.
Algebra i Analiz 12, 1 (2000), 3--69 (Russian);
English translation in
St. Petersburg Math. J. 12, 1 (2001), 1--50.


\bibitem{MR1809664}
Karpenko, N.~A.,
Weil transfer of algebraic cycles.
Indag. Math. (N.S.) 11, 1 (2000), 73--86.

\bibitem{canondim}
Karpenko, N.~A.,
Canonical dimension.
In Proceedings of the International Congress of
Mathematicians.
Volume II (New Delhi, 2010), Hindustan Book Agency,
pp.~146--161.

\bibitem{hypernew-tignol}
Karpenko, N.~A.,
Hyperbolicity of orthogonal involutions.
Doc. Math. Extra Volume: Andrei A. Suslin's Sixtieth Birthday
(2010), 371--392 (electronic).
With an Appendix by Jean-Pierre Tignol.

\bibitem{symple}
Karpenko, N.~A.,
Isotropy of symplectic involutions.
C. R. Math. Acad. Sci. Paris 348, 21-22 (2010), 1151--1153.

\bibitem{outer}
Karpenko, N.~A.,
Upper motives of outer algebraic groups.
In Quadratic forms, linear algebraic groups, and cohomology,
vol.~18 of Dev. Math. Springer, New York, 2010, pp.~249--258.

\bibitem{unitary}
Karpenko, N.~A.,
Hyperbolicity of unitary involutions.
Sci. China Math. 55 (2012), doi: 10.1007/s11425--000--0000--0.

\bibitem{qweil}
Karpenko, N.~A.,
Incompressibility of quadratic Weil transfer of generalized
Severi-Brauer varieties.
J. Inst. Math. Jussieu 11, 1 (2012), 119--131.

\bibitem{gug}
Karpenko, N.~A.,
Unitary grassmannians.
J. Pure Appl. Algebra (2012), doi: 10.1016/j.jpaa.2012.03.024.

\bibitem{upper}
Karpenko, N.~A.,
Upper motives of algebraic groups and incompressibility of
Severi-Brauer varieties.
J. Reine Angew. Math. (Ahead of Print), doi: 10.1515/crelle.2012.011.

\bibitem{MR1632779}
Knus, M.-A., Merkurjev, A., Rost, M. and Tignol, J.-P.,
The Book of Involutions, vol.~44 of American Mathematical
Society Colloquium Publications.
American Mathematical Society, Providence, RI, 1998.
With a preface in French by J.\ Tits.

\bibitem{MR2286826}
Levine, M. and Morel, F.,
Algebraic cobordism.
Springer Monographs in Mathematics. Springer, Berlin, 2007.

\bibitem{MR1850658}
Parimala, R., Sridharan, R. and Suresh, V.,
Hermitian analogue of a theorem of Springer.
J. Algebra 243, 2 (2001), 780--789.

\bibitem{MR2148072}
Vishik, A.,
On the Chow groups of quadratic Grassmannians.
Doc. Math. 10 (2005), 111--130 (electronic).

\bibitem{MR2332601}
Vishik, A.,
Symmetric operations in algebraic cobordism.
Adv. Math. 213, 2 (2007), 489--552.

\bibitem{Vishik-u-invariant}
Vishik, A.,
Fields of {$u$}-invariant {$2^r+1$}.
In Algebra, arithmetic, and geometry: in honor of Yu. I.
Manin. Vol. II.~270 of Progr. Math. Birkh\"auser Boston
Inc., Boston, MA, 2009, pp.~661--685.

\bibitem{MR2377113}
Vishik, A. and Yagita, N.,
Algebraic cobordisms of a {P}fister quadric.
J. Lond. Math. Soc. (2) 76, 3 (2007), 586--604.

\end{thebibliography}

\def\cprime{$'$}

\end{document}